\documentclass{article}
\usepackage{harvard}
\usepackage{graphicx}
\usepackage{camnum}
\usepackage{amsmath}

\usepackage{scalefnt}
\usepackage{tikz,pgfplots}
\usepgfplotslibrary{groupplots}

\allowdisplaybreaks

\newcommand{\ee}{{\mathrm e}}
\newcommand{\ii}{{\mathrm i}}

\newcommand{\ham}{{\cal H}}
\newcommand{\AAA}{\mathcal{A}}
\newcommand{\schr}{Schr\"odinger }
\newcommand{\norm}[2]{\left\| #2 \right\|_{#1}}
\newcommand{\Int}[4]{\int_{#2}^{#3}\!#4\,\mathrm{d}#1}

\newtheorem{example}{Example}
\usepackage{framed}
\usepackage{algorithm}
\usepackage{algorithmic}
\algsetup{indent=2em}
\algsetup{linenosize=\small}
\algsetup{linenodelimiter=:}

\numberwithin{algorithm}{section}

\newcounter{myalg}
\numberwithin{myalg}{section}
\def\openmyalg#1#2{
  \refstepcounter{myalg} \medskip
  {\noindent\bf#1~\themyalg\if#2!{. }\else{ (#2).}\fi} }
\newenvironment{myalg}[1][!]{\openmyalg{Algorithm}{#1}}{\medskip}

\begin{document}
\title{Computation of some dispersive equations through their iterated linearisation}

\author{Guannan Chen\footnote{Department of Mathematical Sciences, University of Bath, email \texttt{gc759@bath.ac.uk}.} \and Arieh Iserles\footnote{Centre for Mathematical Sciences, University of Cambridge, email \texttt{ai@maths.cam.ac.uk}.} \and Karolina Kropielnicka\footnote{Institute of Mathematics, Polish Academy of Sciences, email \texttt{kkropielnicka@impan.pl}.} \and Pranav Singh\footnote{Department of Mathematical Sciences, University of Bath, email \texttt{ps2106@bath.ac.uk}.}}

\thispagestyle{empty}

\maketitle

\begin{abstract}
  It is often the case that, while the numerical solution of the non-linear dispersive equation $\ii\partial_t u(t)=\mathcal{H}(u(t),t)u(t)$ represents a formidable challenge, it is fairly easy and cheap to solve closely related linear equations of the form $\ii\partial_t u(t)=\mathcal{H}_1(t)u(t)+\widetilde{\mathcal H}_2(t)u(t)$, where $\mathcal{H}_1(t)+\mathcal{H}_2(v,t)=\mathcal{H}(v,t)$. In that case we advocate an iterative linearisation procedure that involves fixed-point iteration of the latter equation to solve the former. A typical case is when the original problem is a nonlinear Schr\"odinger or Gross--Pitaevskii equation, while the `easy' equation is linear Schr\"odinger with time-dependent potential.

  We analyse in detail the iterative scheme and its practical implementation, prove that each iteration increases the order, derive upper bounds on the speed of convergence and discuss in the case of nonlinear Schr\"odinger equation with cubic potential the preservation of  structural features of the underlying equation: the $\CC{L}_2$ norm, momentum and Hamiltonian energy. A key ingredient in our approach is the use of the Magnus expansion in conjunction with Hermite quadratures, which allows effective solutions of the linearised but non-autonomous equations in an iterative fashion. The resulting Magnus--Hermite methods can be combined with a wide range of numerical approximations to the matrix exponential. The paper concludes with a number of numerical experiments, demonstrating the power of the proposed approach.
\end{abstract}

\section{Introduction}
\label{sec:intro}
In this paper we concern ourselves with nonlinear dispersive equations in a Hamiltonian form,
\begin{equation}
\label{eq:dispersive}
\ii \partial_t u(t) = \ham(u(t),t)\, u(t), \qquad u(0) = u_0 \in \BB{H},\quad t \geq 0,
\end{equation}
evolving on a complex-valued Hilbert space $\BB{H}$. In many cases of interest the time-dependent and nonlinear Hamiltonian $\ham$,
\begin{equation}
\label{eq:Hamiltonian}
\ham(v,t) = \mathcal{L}(t) + f(v), \qquad f(v) = g(v,\overline{v}),
\end{equation}
comprises of a linear self-adjoint operator on $\BB{H}$, $\mathcal{L}(t)$, and a nonlinear component, $f$, which is a real-valued function.  We assume throughout that $\BB{H}$ is a Sobolev space determined by the nature of the Hamiltonian $\mathcal{H}$ but, to make sense of concepts like the order of a method we require additional regularity, as appropriate.

The main idea underlying the proposed scheme is that it is often considerably easier to solve {\em linear\/} dispersive equations with time-variable potential and to leverage this to the solution of their nonlinear counterparts. We consider the equation  (\ref{eq:dispersive}) written in the form
\begin{equation}
  \label{eq:easy}
  \ii \partial_t u(t)=\ham_1(t)u(t)+\ham_2(u(t),t)u(t),
\end{equation}
such that the equation
\begin{displaymath}
   \ii \partial_t u(t)=\ham_1(t)u(t)+ \widetilde{\ham}_2(t)u(t)
\end{displaymath}
is easy to solve for a reasonable linear Hamiltonian $\widetilde{\ham}_2$. (While this often corresponds to the separation into potential and kinetic parts, this is not necessary for our argument.) In that case we propose the iterative scheme
\begin{Eqnarray}
  \label{eq:IterativeScheme}
  u^{[0]}(t)&=&u_0(t),\\
  \nonumber
  \ii\partial_t u^{[k+1]}(t)&=&\ham_1(t)u^{[k+1]}(t)+\ham_2(u^{[k]}(t),t)u^{[k+1]},\qquad k=0,1,\ldots,
\end{Eqnarray}
where the linear equations are equipped with the same initial and boundary side conditions as (\ref{eq:dispersive}). Note that, even if $\ham(u,t)=\ham(u)$, the linearised equation has a time-dependent potential. This is not a significant issue, provided that we have effective means for a numerical treatment of this case (as we do in the sequel). Note that in the case \R{eq:Hamiltonian} the iterative scheme simplifies to
\begin{displaymath}
  \ii\partial_t u^{[k+1]}(t)=\mathcal{L}(t)u^{[k+1]}(t)+f(u^{[k]}(t))u^{[k+1]}(t),\qquad k=0,1,\ldots.
\end{displaymath}

We terminate the iteration after a suitable number of steps. Needless to say, even taking for granted that (\ref{eq:easy}) is easy to solve repeatedly, the first relevant issue is the convergence of (\ref{eq:IterativeScheme}) to the exact solution of (\ref{eq:dispersive}) and its speed. The second issue, in the spirit of geometric numerical integration \cite{hairer06gni}, is whether important geometric features of (\ref{eq:dispersive}), e.g.\ Hamiltonicity or $\CC{L}_2$ conservation, are preserved by the scheme (\ref{eq:IterativeScheme}).

It is important to mention that our method is {\em not\/} an operatorial splitting \cite{blanes24smd,mclachlan02sm}, an approach ubiquitous in modern numerical analysis of PDEs, since we do not compose the solution from a split form of (\ref{eq:easy}), $\ii\partial_t u=\ham_1(t)u$ and $\ii\partial_t u=\ham_2(u,t)u$.  It is perhaps more illuminating to consider our approach as an iterative linearisation of the PDE.

The proposed schemes will be illustrated in this paper with the example of the cubic nonlinear \schr (NLS) equation,
\begin{equation}
\label{eq:NLS}
\ii \partial_t u(t) = -\Delta u(t) + \lambda |u(t)|^2 u(t), \qquad u(0) = u_0 \in \BB{H},\quad t \geq 0, \quad \lambda \in \BB{R},
\end{equation}
as well as the Gross--Pitaevskii (GP) equation,
\begin{equation}
\label{eq:GP}
\ii \partial_t u(t) = (-\Delta + V^0  + V^{\mathrm{e}}(t)) u(t) + \lambda |u(t)|^2 u(t), \qquad u(0) = u_0 \in \BB{H},\quad t \geq 0, \quad \lambda \in \BB{R},
\end{equation}
where it is typical to consider $\BB{H} = \CC{H}^1(\BB{K};\BB{C})$ and either $\BB{K} = \BB{R}^d$ or $\BB{K} = \BB{T}^d$. (In reality we implicitly assume higher regularity, otherwise our discussion of order and speed of convergence in Section~3 need be amended in a transparent manner.) A negative value of $\lambda$ corresponds to a focusing equation, while $\lambda>0$ to a defocusing one. In the case of the cubic NLS example, $\mathcal{L}(t) = -\Delta$, $g(v,w) = \lambda \langle v, w \rangle$ and $f(v) = \lambda |v|^2$. The linearisation (\ref{eq:easy}) splits the kinetic part $\ham_1(u,t)=-\Delta$ and the potential part $\ham_2(u,t)=\lambda|u|^2$.
In the case of the Gross--Pitaevskii (GP) equation, $\mathcal{L}(t) = -\Delta +V^0(x) + V^{\mathrm{e}}(x,t)$ for some real-valued functions $V^0$ (typically a confining potential) and $V^{\mathrm{e}}$ (typically an external time-dependent potential).

The iteratively linearised schemes developed in this work can be combined with a wide range of methods for approximating the matrix exponential. In Section~\ref{sec:numerics} we present numerical examples to demonstrate the combination of our schemes with splitting methods and Krylov subspace methods.

\setcounter{equation}{0}
\setcounter{figure}{0}
\section{The algorithm}
\label{sec:algorithm}

\subsection{Iterated linearisation}

It is a good idea to recall that we are solving the dispersive PDE
\begin{displaymath}
  \ii\partial_t u(t)=\ham(u(t),t)u(t),\qquad u(0)=u_0,
\end{displaymath}
by the iterative procedure
\begin{displaymath}
  \ii\partial_t u^{[k+1]}(t)=\ham_1(t)u^{[k+1]}(t)+\ham_2(u^{[k]}(t),t)u^{[k+1]},\qquad k=0,1,\ldots,
\end{displaymath}
with the initial condition $u_0$ and the same boundary conditions as the original PDE. Formally, the later defines functional iteration,
\begin{equation}
  \label{eq:functional}
  u^{[k+1]}=G(u^{[k]})
\end{equation}
in the ambient Hilbert space $\BB{H}$ and our first concern is with the convergence of this procedure from the initial guess $u^{[0]}=u_0$. In other words, we wish to ascertain whether a fixed point $\hat{u}$ of \R{eq:functional}, which clearly is a solution of \R{eq:dispersive}, is attractive and whether $u_0$ belongs to its basin of attraction. In practical terms, we wish to demonstrate that $G$ is a contraction mapping in a closed neighbourhood of $\hat{u}\in\BB{H}$, i.e.\ $\CC{Lip}_{\bar{B}_r(\hat{u})}(G) < 1$ for some $r>0$, where $\bar{B}_r(w)\subset\BB{H}$ is a closed ball of radius $r$ about $w$. Note that this might require a restriction upon the time step so that $h\leq r$.

In practice the PDE \R{eq:IterativeScheme} is computed by a sequence of existing solvers for linear dispersive equations, $\mathcal{S}_k$,
\begin{displaymath}
  u^{[k+1]}(t_{n+1}) = \mathcal{S}_k(u^{[k]},u(t_n)), \quad k \geq 0,
\end{displaymath}
either for a fixed number of iterations $k = 0, \ldots, K-1$ or terminating when the update in step $k+1$, $\norm{\CC{L}^2}{u^{[k+1]}(t_n) - u^{[k]}(t_n)} \leq \delta$, is smaller than some threshold $0<\delta\ll 1$. In other words, $\mathcal{S}_k$ is a discretisation of the solution map $G$ in step $k$. We allow the dependence of the discrete solution operator upon $k$ because for different iterations we might use, for example, solvers of different accuracy.

For the cubic NLS (\ref{eq:NLS}), for instance, we solve a sequence of linear \schr equations,
\begin{equation}
\label{eq:lin_NLS}
\ii \partial_t u^{[k+1]}(t) = -\Delta u^{[k+1]}(t) + V^{\mathrm{nl}, {[k]}}(t) u^{[k+1]}(t), \qquad u^{[k+1]}(t_n) = u(t_n),
\end{equation}
where
\begin{displaymath}
  V^{\mathrm{nl}, {[k]}}(t) = \lambda |u^{[k]}(t)|^2,
\end{displaymath}
acts as a time-dependent potential.

\subsection{The Magnus expansion}
The solution of {\em linear\/} dispersive equations with time-dependent Hamiltonians can be approximated effectively using the Magnus expansion \cite{blanes16cig,iserles00lgm,hochbruck2003magnus,kormann2008accurate},. Specifically, $u(t_{n+1})\approx \ee^{h\Theta(t_{n+1},t_n)}u(t_n)$ where, for example, a fourth-order Magnus expansion for the linear dispersive equation
\[\ii \partial_t u(t) = \AAA(t) u(t), \]
with time-dependent vector field $\AAA(t)$ is
\begin{equation}
\label{eq:MagnusO2}
\Theta_2(t_{n+1},t_n) = -\ii \Int{\zeta}{0}{h}{\AAA(t_n+\zeta)} + \Frac12\Int{\zeta}{0}{h}{\Int{\xi}{0}{\zeta}{ [\AAA(t_n+\xi),\AAA(t_n+\zeta)]  }   }.
\end{equation}
Supposing \R{eq:Hamiltonian},  the  PDE \R{eq:IterativeScheme} is linear,
\begin{displaymath}
  \ii\partial_t u^{[k+1]}(t)=[\mathcal{L}(t)+f(u^{[k]}(t),t)] u^{[k+1]},
\end{displaymath}
and we may set $\AAA(t)=\AAA_k(t):={\mathcal L}(t)+f(u^{[k]}(t),t)$. Using the Magnus expansion, the solution can be approximated as
\begin{equation}
\label{eq:Mag_uk}
u^{[k+1]}(t_{n+1}) = \exp(\Theta_2(t_{n+1},t_n; u^{[k]}))\, u(t_n), \qquad k \geq 0,
\end{equation}
where  $u^{[k]}$ indicates explicitly the dependence of the vector field on the previous iterate.

The expression \R{eq:Mag_uk} is not a `real' numerical method because in practice it requires a semidiscretisation of $\AAA$, computation of the integrals in \R{eq:MagnusO2} and the computation (or approximation) of matrix exponential. Each of these steps might be costly -- cf.\ \cite{bader14eas} for a more detailed discussion -- and ideally we need additional ideas to bring down the expense of the computation.

If the linear self-adjoint operator $\mathcal{L}(t)$ can be decomposed into time-independent and time-dependent components, $\mathcal{L}(t) = \mathcal{L}_0 + \mathcal{L}_1(t)$,
the commutator in \R{eq:MagnusO2} can be simplified as
\[ [\AAA(t+\xi),\AAA(t+\zeta)] = [\mathcal{L}_0, \mathcal{L}_1(t+\zeta) - \mathcal{L}_1(t+\xi)] + [\mathcal{L}_0, f(u^{[k]}(t+\zeta)) - f(u^{[k]}(t+\xi))]. \]

The Magnus expansion \R{eq:MagnusO2} reduces to
\begin{eqnarray}
\nonumber
\Theta_2(t_{n+1},t_n) &\!\!\! = \!\!\!& -\ii\, \left(h \mathcal{L}_0 + \mu_{0,0}^{\mathcal{L}_1}(t_{n+1},t_n) + \mu_{0,0}^{f \circ u^{[k]}}(t_{n+1},t_n)\right) \\
\label{eq:MagnusO2simplified}&\!\!\!\!\!\! &\mbox{} + [\mathcal{L}_0, \mu_{1,1}^{\mathcal{L}_1}(t_{n+1},t_n)]+ [\mathcal{L}_0, \mu_{1,1}^{f \circ u^{[k]}}(t_{n+1},t_n)],
\end{eqnarray}
where $\mu_{j,k}$ are the line integrals defined in \cite{iserles19csf,iserles19sse},
\begin{displaymath}
\mu_{j,k}^\mathcal{P}(t_{n+1},t_n) = \Int{\zeta}{0}{h}{\tilde{B}_j^k(h,\zeta) \mathcal{P}(t+\zeta)},
\end{displaymath}
and $\tilde{B}_j(h,\zeta) = h^j B_j(\zeta/h)$ is the rescaled $j$th Bernoulli polynomial. It is known that $\mu_{0,0} = \O{h}$ and $\mu_{1,1} = \O{h^3}$.

In the context of the linearised \schr equations \R{eq:lin_NLS} that appear in the iterative solution of the cubic NLS \R{eq:NLS},
$\mathcal{L}_0 = -\Delta$ and $\mathcal{L}_1\equiv0$. In the context of the Gross--Pitaevskii equation, $\mathcal{L}_0 = -\Delta + V^0(x)$ and $\mathcal{L}_1(t) = V^{\mathrm{e}}(x,t)$ is the external time-dependent potential.

\subsection{Computing with information at the endpoints}

The integrals in \R{eq:MagnusO2} can be computed to fourth order using just two Gaussian points in $[t_n,t_{n+1}]$ \cite{iserles00lgm} but this takes no account of the iterative nature of \R{eq:IterativeScheme}. To explain why, consider the second iteration $u^{[2]}$: its computation means that we need to evaluate $u^{[1]}$ at the two Gaussian points, $t_n+c_\pm h$, where $c_\pm=\frac12\pm\frac{\sqrt{3}}{6}$. This means, though, that we need {\em three\/} values of $u_0$: $t_n+c_-^2h$ and $t_n+c_-c_+h$ to approximate $u^{[1]}(t_n+c_-h)$, while $t_n+c_-c_+h$ (again) and $t_n+c_+^2h$ are required to approximate $u^{[1]}(t_n+c_+h)$. In the general case, it is easy to ascertain that we need $k+1$ function evaluations of $u_0$ to approximate $u^{[k]}$ to fourth order. This clearly is suboptimal and we wish to avoid this unwelcome linear growth.\footnote{This becomes even worse once we seek higher-order methods. We need three function evaluations to compute all Magnus integrals to order 6 -- the formula is more complicated than \R{eq:MagnusO2} \cite{iserles00lgm} -- but this is translated into ${{k+2}\choose 2}$ evaluations of $u_0$ to compute $u^{[k]}$.}


The remedy to this state of affairs if to use {\em Hermite quadrature\/} \cite{micchelli73qfh}, employing just the values of the integrand {\em and its derivatives\/} at the endpoints. To attain order~4  in both integrals, we need to know just the integrand and its derivative at $t_n$ and $t_{n+1}$: altogether $2K$ evaluations for  $K$ iterations.

The fourth-order Magnus expansion \R{eq:MagnusO2} requires  a fourth-order approximation of the integrals. This can be achieved via  Hermite quadrature utilising only the Hamiltonian and its time-derivative, $\ham(t_{n})$, $\ham(t_{n+1})$, $\partial_t \ham(t_{n})$ and $\partial_t \ham(t_{n+1})$, at the two end points of the time interval $[t_n,t_{n+1}]$. The main idea, following \cite{iserles99sld,iserles00lgm,zanna98ons}, is to form a Hermite interpolation polynomial (in the present case a cubic) $\mathcal{B}$, say,  replace $\AAA$ by $\mathcal{B}$ in \R{eq:MagnusO2} and integrate exactly the ensuing polynomial expressions.

More concretely, this can be directly applied to the approximation of the integrals in \R{eq:MagnusO2simplified}. In particular, since $\mathcal{L}_1$ is expected to be known for all $t \in [t_n,t_{n+1}]$, a variety of quadrature approximations can be utilised for computing $\mu_{1,1}^{\mathcal{L}_1}(t_{n+1},t_n)$ up to the required accuracy of $\O{h^5}$. In contrast, $u^{[k]}$ is not known at all $t \in [t_n,t_{n+1}]$ and we use the fourth-order Hermite approximations,
\begin{Eqnarray*}
 \mu_{0,0}^{\mathcal{P}}(t_{n+1},t_n) & \approx & \frac{h}{2} \left(\mathcal{P}(t_n)+ \mathcal{P}(t_{n+1})\right) + \frac{h^2}{12} \left(\partial_t \mathcal{P}(t_n) - \partial_t \mathcal{P}(t_{n+1})\right),\\
 \mu_{1,1}^{\mathcal{P}}(t_{n+1},t_n) & \approx & \frac{h^2}{2} \left(\mathcal{P}(t_{n+1})- \mathcal{P}(t_{n})\right) - \frac{h^3}{120} \left(\partial_t \mathcal{P}(t_n) + \partial_t \mathcal{P}(t_{n+1})\right),
\end{Eqnarray*}
for approximating $\mu_{0,0}^{f \circ u^{[k]}}$ and $\mu_{1,1}^{f \circ u^{[k]}}$.

{\bf General dispersive PDE.} For a general dispersive PDE \R{eq:dispersive} the nonlinear function is assumed to be of the form $f(v) = g(v,\bar{v})$ and the derivative can be computed as
\begin{Eqnarray*}
 \partial_t (f(u^{[k]}(t)) &=&  \partial_t (g(u^{[k]}(t),\overline{u^{[k]}(t)})\\
 & = & \CC{D}_1 g(u^{[k]}(t),\overline{u^{[k]}(t)})\, \partial_t u^{[k]}(t) +  \CC{D}_2 g(u^{[k]}(t),\overline{u^{[k]}(t)})\, \overline{\partial_t u^{[k]}(t)}\\
 & = & - \ii \left( \CC{D}_1 g(u^{[k]}(t),\overline{u^{[k]}(t)}) - \CC{D}_2 g(u^{[k]}(t),\overline{u^{[k]}(t)})\right)\! \mathcal{H}(u^{[k]}(t), t) u^{[k]}(t).
\end{Eqnarray*}
{\bf The cubic nonlinear Schr\"odinger (NLS) and Gross--Pitaevskii (GP) equations.}
In the case of cubic nonlinearity in the NLS or GP equations,
$f(v) = \lambda |v|^2$, $g(v,w) = \lambda \langle v, w \rangle$, $\CC{D}_1 g(v,w) = \lambda  \CC{L}_w$ and $\CC{D}_2 g(v,w) = \lambda  \CC{L}_v$, where $\CC{L}_w$
is the (conjugate) linear functional $\bar{\CC{L}}_w(u) = \langle u, w \rangle$, while  $\CC{L}_v$ is the linear functional $\bar{\CC{L}}_v(u) =  \langle v, u \rangle$.
Consequently,
\begin{Eqnarray*}
\nonumber \partial_t (f(u^{[k]}(t)) &=&  2 \lambda \Re \langle u^{[k]}(t), \partial_t u^{[k]}(t)\rangle \\
& = & - 2 \ii \lambda \Re \left\langle u^{[k]}(t), \left(\mathcal{L}(t) +  \lambda |u^{[k]}(t)|^2\right)\, u^{[k]}(t)\right\rangle,
\end{Eqnarray*}
which can be readily evaluated at $t_n$ and $t_{n+1}$.
The distinction between the NLS and GP equations is that $\mathcal{L}(t) = -\Delta$ in the case of the NLS, whereas $\mathcal{L}(t) = -\Delta + V^0(x) + V^{\mathrm{e}}(x,t)$ in the case of GP.

\subsubsection{Pseudocode for cubic NLS and GP}

A schematic pseudocode for our algorithm, as applied to the cubic NLS and Gross--Pitaevskii equation, is described in Algorithm~\ref{alg:GP}.
\begin{framed}
\begin{myalg}[Iterative solver for cubic NLS and GP]
  \label{alg:GP}
  \begin{algorithmic}[1]
    \REQUIRE{$\lambda, T, N$ and $u_0 \in \mathcal{H}$}
    \STATE $h=T/N$
    \FOR{$n = 0, 1, 2, \dots, N$}

        \STATE $V^{\mathrm{nl}}_n = \lambda |u_n|^2$

         $v_n = (-\Delta + V^0 + V^{\mathrm{e}}(t_n) + V^{\mathrm{nl}}_n)\, u_n$

         $dV^{\mathrm{nl}}_n = - 2 \ii \lambda \Re \langle u_n, v_n \rangle$

        \STATE $\mu^{\mathrm{e}}_{0,0}, \mu^{\mathrm{e}}_{1,1} = \mathrm{Quadrature}(V^{\mathrm{e}}(t), t_n, h)$

        \STATE $u^{[0]}_{n+1} = u_n$

        \FOR{$k= 0, 1, \dots K-1$}

        \STATE $V^{\mathrm{nl}}_{n+1} = \lambda |u_{n+1}^{[k+1]}|^2$

         $v_{n+1} = (-\Delta + V^0 + V^{\mathrm{e}}(t_{n+1}) + V^{\mathrm{nl}}_{n+1})\, u_{n+1}^{[k+1]}$

         $dV^{\mathrm{nl}}_{n+1} = - 2 \ii \lambda \Re \langle u_{n+1}^{[k+1]}, v_{n+1} \rangle$

        \STATE $\mu^{\mathrm{nl}}_{0,0}, \mu^{\mathrm{nl}}_{1,1} = \mathrm{Hermite}(V^{\mathrm{nl}}_n, V^{\mathrm{nl}}_{n+1}, \partial_t V^{\mathrm{nl}}_n, \partial_t V^{\mathrm{nl}}_{n+1})$

        \STATE $\mu_{0,0} = \mu^{\mathrm{e}}_{0,0} + \mu^{\mathrm{nl}}_{0,0}$, $\quad \mu_{1,1} = \mu^{\mathrm{e}}_{1,1} +  \mu^{\mathrm{nl}}_{1,1}$

        \STATE $\Theta_2 = \ii h \Delta - \ii h V^0 -\ii \mu_{0,0} - \left[\Delta,  \mu_{1,1} \right]$

        \STATE $u^{[k+1]}_{n+1} = \exp\left(\Theta_2\right) u_n$
        \ENDFOR
        \STATE $u_{n+1} = u^{[K]}_{n+1}$
    \ENDFOR
  \end{algorithmic}
\end{myalg}
\end{framed}

For most problems of interest $K=3$  suffices for fourth-order  convergence, cf.\ Theorem~3. Note that any reasonable quadrature can be used to compute $\mu^{\mathrm{e}}_{0,0}$ and $\mu^{\mathrm{e}}_{1,1}$.














\subsection{Approximation of the exponential}

The exponential of the Magnus expansion  $\Theta_2$, \eqref{eq:MagnusO2simplified}, in Algorithm~\ref{alg:GP} only needs to be computed to a fourth-order accuracy.
This can be achieved using a wide range of methods for matrix exponential \cite{moler1978nineteen}. Among these, some of the most widely utilised methods are splitting methods \cite{mclachlan02sm,blanes24smd} and Krylov subspace methods \cite{saad2003iterative}.


Splitting methods are particularly useful when the exponent can be split into two parts, which are easy to exponentiate separately. In the context of the cubic NLS and GP equations, the exponent $\Theta_2$ is a sum of a Laplacian, a multiplication operator, and a commutator involving a Laplacian and a multiplicative operator.

In the absence of the commutator, splitting methods could be effectively used to compute the exponential of the Magnus expansion \R{eq:MagnusO2simplified}. However, the presence of the commutator makes this infeasible. The presence of the commutator can also make Krylov iterations more expensive. However, if required, this commutator can be eliminated with minimal additional cost.

We note that if $B$ and $C$ commute, a simple application of the Baker--Campbell--Hausdorff formula shows that
\[ \exp(h A + h B + h^3 [A,C]) = \ee^{-h^3 C} \ee^{h A + h B} \ee^{h^3 C} + \O{h^5}.\]
Consequently,
\begin{equation}
\label{eq:MagnusAssymetricSplit}
\exp(\Theta_2(t_{n+1},t_n)) = \ee^{-\mu_{1,1}}\ \ee^{-\ii h \mathcal{L}_0 - \ii \mu_{0,0}}\ \ee^{\mu_{1,1}} + \O{h^5},
\end{equation}
where
\begin{Eqnarray}
\label{eq:mu00} \mu_{0,0} & = & \mu_{0,0}^{\mathcal{L}_1}(t_{n+1},t_n) - \mu_{0,0}^{f \circ u^{[k]}}(t_{n+1},t_n)\\
\label{eq:mu11} \mu_{1,1} & = & \mu_{1,1}^{\mathcal{L}_1}(t_{n+1},t_n) - \mu_{1,1}^{f \circ u^{[k]}}(t_{n+1},t_n).
\end{Eqnarray}

The central exponential in \R{eq:MagnusAssymetricSplit} is now free of commutators and can be approximated by using any method for computing the matrix exponential. In particular, any fourth-order splitting can be employed, provided exponentials of $\mathcal{L}_0$, $\mathcal{L}_1$ and $f(v)$ are readily computable: this is contingent upon their structure following discretisation. Due to linearity of the integral, $\mu_{j,k}^{\mathcal{L}_1}$ and $\mu_{j,k}^{f \circ u^{[k]}}$ should have the same structure as $\mathcal{L}_1$ and $f$, respectively.

In the context of cubic NLS and Gross--Pitaevskii equations, where $\mathcal{L}_1(t)$ is a function, we may also approximate the central exponential by utilising {\em compact} fourth-order splittings such as \cite{chin05fsi} and \cite{omelyan2003symplectic} since the first component $\mathcal{L}_0$ is a Laplacian and the second component $\mu_{0,0}$ is a multiplication operator.













\setcounter{equation}{0}
\setcounter{figure}{0}
\section{Convergence}
\label{sec:convergence}

It is convenient to assume in this section that the underlying domain is $\BB{T}^d$, noting that a generalisation to many other boundary condv(zero Dirichlet or Neumann conditions in a compact domain, a Cauchy problem in $\BB{R}^d$) is straightforward. We  further assume in this section that the Hamiltonian can be written in the form $H(u)=\mathcal{G}(u)u$, understanding that, for example,  $\mathcal{G}(u)=-\Delta+\lambda |u|^2$ needs to be taken for the NLS (1.3).

Our starting point is the equation (1.1), where we assume that $u_0\in\CC{H}^s(\BB{T}^d)$ for some $s\in\BB{N}$ and $\mathcal{L}$ acts weakly on $\CC{H}^s(\BB{T}^d)$. We note that the self-adjointness of the Hamiltonian implies that
\begin{equation}
  \label{eq:3.1}
  \Im\langle \mathcal{G}(u)u,u\rangle=0,\qquad u\in \CC{H}^s(\BB{T}^d),
\end{equation}
where $\langle\,\cdot\,,\,\cdot\,\rangle$ is the standard $\CC{L}_2(\BB{T}^d)$ inner product. An immediate consequence is that
\begin{displaymath}
  \frac12 \frac{\D\|u\|^2_{\CC{L}_2}}{\D t}=\Re \langle \partial_t u,u\rangle=-\Re \ii \langle \mathcal{G}(u)u,u\rangle =\Im\langle \mathcal{G}(u)u,u\rangle=0
\end{displaymath}
and the $\CC{L}_2(\BB{T}^d)$ norm is preserved for all $t\geq0$.

The concern of this section is with the convergence of the iterative algorithm (2.2). We let
\begin{displaymath}
  e^{[k]}(x,t)=u^{[k]}(x,t)-u(x,t),\qquad \rho_k(t)=\|e^{[k]}(\,\cdot\,,t)\|_{\CC{L}_2}.
\end{displaymath}

\begin{lemma}
  For every $t\in[t_n,t_{n+1}]$ it is true that
  \begin{equation}
    \label{eq:3.2}
    \rho_{k+1}(t)\leq h\|\mathcal{G}(u^{[k]})-\mathcal{G}(u)\|_{\CC{L}_2}\cdot\,\|u\|_{\CC{L}_2},
  \end{equation}
  where $h=t_{n+1}-t_n$.
\end{lemma}

\begin{proof}
  It is an immediate consequence of (1.1) and (2.2) that
  \begin{displaymath}
    \ii \partial_t e^{[k+1]}=\mathcal{G}(u^{[k]})e^{[k+1]}+[\mathcal{G}(u^{[k]})-\mathcal{G}(u)]u.
  \end{displaymath}
  Similarly to \cite{hochbruck03omi}, we observe that
  \begin{displaymath}
    \frac{\D \|e^{[k+1]}\|^2_{\CC{L}_2}}{\D t}=2\|e^{[k+1]}\|_{\CC{L}_2}\frac{\D\|e^{[k+1]}\|_{\CC{L}_2}}{\D t}=2\rho_{k+1}(t)\rho_{k+1}'(t),
  \end{displaymath}
  on the other hand
  \begin{displaymath}
    \frac{\D \|e^{[k+1]}\|^2_{\CC{L}_2}}{\D t}=\frac{\D}{\D t} \langle e^{[k+1]},e^{[k+1]}\rangle=2\Re \left\langle e^{[k+1]},\frac{\D e^{[k+1]}}{\D t}\right\rangle\!.
  \end{displaymath}
  Comparing the two expressions,
  \begin{Eqnarray*}
    \rho_{k+1}'&=&\frac{1}{\rho_{k+1}} \Re \left\langle e^{[k+1]},\frac{\D e^{[k+1]}}{\D t}\right\rangle\\
    &=&\frac{1}{\rho_{k+1}} \Re\left\langle e^{[k+1]},-\ii \mathcal{G}(u^{[k]})e^{[k+1]}-\ii[\mathcal{G}(u^{[k]}-\mathcal{G}(u)]u\right\rangle\!.
  \end{Eqnarray*}
  However, \R{eq:3.1} implies that $\Re \langle e^{[k+1]},\mathcal{G}(u^{[k]})e^{[k+1]}\rangle=0$, therefore
  \begin{displaymath}
    \rho_{k+1}'=-\frac{\ii}{\rho_{k+1}} \langle e^{[k+1]},[\mathcal{G}(u^{[k]}-\mathcal{G}(u)]u\rangle.
  \end{displaymath}
  Moreover, $\rho_{k+1}(t_n)=0$ and integration yields
  \begin{displaymath}
    \frac12 \rho_{k+1}^2=-\ii \int_{t_n}^t \langle e^{[k+1]}(\,\cdot\,,\tau),[\mathcal{G}(u^{[k]}(\,\cdot\,,\tau))-\mathcal{G}(u(\cdot\,,\tau))]u(\,\cdot\,,\tau)\rangle\D\tau
  \end{displaymath}
  for $t\in[t_n,t_{n+1}]$. Simple application of the Cauchy--Schwartz inequality results in \R{eq:3.2} and concludes the proof.
\end{proof}

Our next goal is to identify the {\em order\/} of each iteration, i.e.\ the number (if it exists) $p_k\in\BB{N}$ such that
\begin{displaymath}
  \rho_k(t)\leq c_k h^{p_k+1},\qquad t\in[t_n,t_{n+1}]
\end{displaymath}
for some $c_k>0$. Were we to prove that the $p_k$s form a strictly increasing sequence and the $c_k$s form at most geometrically increasing sequence (i.e., $\limsup_{k\rightarrow\infty} k^{-1}\log c_k<\infty$), the convergence of the iterative scheme (2.2) for sufficiently small $h>0$ will be assured.

We assume that $\mathcal{G}$ is Fr\'echet differentiable and denote its Fr\'echet derivative by $\mathcal{G}'$. Note that, $\mathcal{G}$ being a differential operator, this is a somewhat delicate assumption. For example, for the Schr\"odinger operator (1.3), namely $\mathcal{G}(u)=-\Delta+\lambda|u|^2$, the potential part $\lambda|u|^2$ presents no problem but the kinetic part, $-\Delta$, requires more consideration. The Fr\'echet derivative of $-\Delta$ is the least number $\gamma>0$ such that $\|\Delta v\|_{\CC{L}_2}\leq \gamma\|v\|_{\CC{L}_2}$: this is a consequence of $-\Delta$ being a linear operator. Finding $\gamma$ is easy, provided that everything lives in $\CC{H}^2(\BB{T}^d)$ -- in that case
\begin{displaymath}
  \|\Delta v\|_{\CC{L}_2}^2=\langle \Delta v,\Delta v\rangle_{\CC{L}_2}=\langle v,\Delta^2v\rangle_{\CC{L}_2}\leq \|\Delta^2\|_{\CC{L}_2}\|v\|_{\CC{L_2}}
\end{displaymath}
and maximising over $v\in\CC{H}^2(\BB{T}^d)\setminus\{0\}$ results in $\gamma=\|\Delta^2\|^{1/2}_{\CC{L}_2}$. Taking $\mathcal{H}=\CC{H}^1(\BB{T}^d)$, which is typically sufficient for the consideration of the equation (1.3) in its weak form, falls short of the requirements of our convergence analysis.

Since
\begin{displaymath}
  \|\mathcal{G}(v+\delta)-\mathcal{G}(v)\|_{\CC{L}_2}\leq \mathcal{G}'(v)\|\delta\|_{\CC{L}_2},\qquad v,\delta\in\CC{H}^2(\BB{T}^d),\quad \|\delta\|_{\CC{L}_2}\ll1,
\end{displaymath}
we deduce that
\begin{displaymath}
  \|\mathcal{G}(u^{[k]})-\mathcal{G}(u)\|_{\CC{L}_2}=\|\mathcal{G}(u+e^{[k]})-\mathcal{G}(u)\|_{\CC{L}_2}\leq \mathcal{G}'(u)\rho_k.
\end{displaymath}
It now follows from \R{eq:3.2} by induction that
\begin{displaymath}
  \rho_{k+1}(t)\leq h\mathcal{G}'(u)\|u\|_{\CC{L}_2}\rho_k(t)\leq \cdots\leq h^k \left[\mathcal{G}'(u)\|u\|_{\CC{L}_2}\right]^k \rho_1(t).
\end{displaymath}

\begin{theorem}
  \label{Convergence}
  Given that $\mathcal{G}$ is Fr\'echet differentiable in $\BB{T}^d$, it is true that
  \begin{equation}
     \label{eq:3.3}
     \rho_k(t)\leq \left[\mathcal{G}'(u)\|u\|_{\CC{L}_2}\right]^{k-1} c_1 h^{p_1+k},\qquad k\in\BB{N}
  \end{equation}
  and the method (2.2) converges for $h<[\mathcal{G}'(u)\|u\|_{\CC{L}_2}]^{-1}$.
\end{theorem}

\begin{proof}
  Since $c_k\leq \left[\mathcal{G}'(u)\|u\|_{\CC{L}_2}\right]^{k-1} c_1$, the $c_k$s form a geometric sequence and convergence follows.
\end{proof}

The inductive argument in the theorem can be commenced from $k=0$, whereby $p_0=0$ implies immediately that $p_k\geq k$. In other words, {\em each iteration increases the order by at least one unit!\/} The clear practical implication is that, once the goal is to attain specific temporal order, we need to iterate (2.2) certain number of times, thereby stoping, without any convergence tests.

Note four important points:
\begin{itemize}
\item The notion of order in Theorem~\ref{Convergence} is different from the standard concept of order in time-stepping schemes, which makes sense only in the presence of adequate regularity because the error of time-stepping schemes (at any rate, those that fit into B-series formalism \cite{hairer06gni}) is expressed in terms of high derivatives and elementary differentials. Order in the sense of Theorem~\ref{Convergence}, though, is independent of regularity, as long as the ambient space is smooth enough for $\mathcal{G}'(u)$ to be bounded.
\item The operator $\mathcal{G}$ need not be bounded: this is precisely the reason why we employ Fr\'echet derivatives (which remain bounded locally, for specific function $u$), rather than operator norms.
\item The statement of Theorem \ref{Convergence} is about an {\em exact\/} solution of linear equations in (2.2), while practical computation requires the use of numerical methods, of the kind to which we have alluded in Section~1. The order of the method, however, remains the same as long as the order of the `linear' method used in solving for $u^{[k]}$, is at least $p_k$.
\item Since in general the higher the order of a time-stepping method, the greater its computational cost, we conclude that it is considerably more efficient to vary the order of the `linear' method in the course of the iteration: done right, this reduces the cost without affecting the rate of convergence.
\end{itemize}

However, for a large set of problems (1.1) we can do better than that! Suppose thus that $\mathcal{G}$ is a function of $|u|$ only -- with a minor abuse of notation we write $H(u)=\mathcal{G}(|u|)u$. As in (1.2), $H(v)=\mathcal{L}+f$, except that now, abusing notation again, $f=f(|v|)$. For a want of a better name, we might term such problems as being of {\em Schr\"odinger type.\/}

We solve the equation
\begin{equation}
  \label{eq:3.4}
  \ii\partial_t u=\mathcal{L} u+f(|u|)u,\qquad u(x,t_n)=u^{[0]}(x,t_n),
\end{equation}
using the {\em Strang splitting\/} \cite{hairer06gni},
\begin{Eqnarray}
  \nonumber
  \ii \partial_t v_1&=&\mathcal{L}v_1,\qquad\qquad\, v_1(x,t_n)=u^{[0]}(x,t_n),\\
  \label{eq:3.5}
  \ii\partial v_2&=&f(|v_2|)v_2,\qquad v_2(x,t_n)=v_1(x,t_n+\Frac12 h),\\
  \nonumber
  \ii\partial_t v_3&=&\mathcal{L}v_3,\qquad\qquad\, v_3(x,t_n+\Frac12 h)=v_2(x,t_n+\Frac12 h),\\
  \nonumber
  u(x,t_{n+1})&=&v_3(x,t_{n+1}).
\end{Eqnarray}
The second step in the Strang splitting is an ODE. Consider though the solution of the ODE
\begin{displaymath}
  \ii y'=f(|y_0|)y,\qquad y(t_n)=y_0
\end{displaymath}
-- this is a linear equation whose exact solution is $y(t)=\exp(-\ii t f(|y_0|))y_0$. In particular, $|y(t)|\equiv |y_0|$, therefore $f(|y|)=f(|y_0|)$ and the solution of the linear ODE is also the solution of
\begin{displaymath}
  \ii y'=f(|y|)y,\qquad y(t_n)=y_0.
\end{displaymath}
In other words, the solution of \R{eq:3.5} is {\em identical\/} to that of
\begin{Eqnarray}
  \nonumber
  \ii \partial_t v_1&=&\mathcal{L}v_1,\qquad\qquad\qquad\quad\! v_1(x,t_n)=u^{[0]}(x,t_n),\\
  \label{eq:3.6}
  \ii\partial v_2&=&f(|v_2(x,t_n)|)v_2,\qquad v_2(x,t_n)=v_1(x,t_n+\Frac12 h),\\
  \nonumber
  \ii\partial_t v_3&=&\mathcal{L}v_3,\qquad\qquad\qquad\quad\! v_3(x,t_n+\Frac12 h)=v_2(x,t_n+\Frac12 h),\\
  \nonumber
  u(x,t_{n+1})&=&v_3(x,t_{n+1}).
\end{Eqnarray}
The system \R{eq:3.6}, however, is exactly what we obtain by applying Strang splitting in the first step of the iterative algorithm (2.2), applied to the PDE \R{eq:3.4}. To rephrase, $u^{[1]}$ is nothing else but the Strang-splitting solution of \R{eq:3.4}. Since Strang splitting is a second-order method, we deduce that $p_1=2$. This explains why in the proof of Theorem~\ref{Convergence} we commenced the iterative argument from $k=1$ and it is worthwhile to formulate the new result as a theorem.

\begin{theorem}
   \label{Convergence1}
   Let $H(u)=\mathcal{G}(|u|)u$, where $\mathcal{G}(v)=\mathcal{L}+f(v)$, $\mathcal{L}$ is a self-adjoint linear differential operator and $f$ a real-valued scalar function. Provided that the first step of the algorithm (2.2) is solved by the Strang splitting, it is true that $p_k=k+1$, $k\in\BB{N}$ and in place of \R{eq:3.3} we have the error bound
   \begin{displaymath}
     \rho_k(t)\leq \left[\mathcal{G}'(u)\|u\|_{\CC{L}_2}\right]^{k-1} c_1 h^{k+2},\qquad k\in\BB{N}.
   \end{displaymath}
   As before, convergence is assured provided that $h<[\mathcal{G}'(u)\|u\|_{\CC{L}_2}]^{-1}$.
\end{theorem}

\setcounter{equation}{0}
\setcounter{figure}{0}
\section{Preservation of the $\CC{L}_2$ norm, momentum and Hamiltonian energy in cubic \schr equation}
\label{sec:energy}
As in the heart of geometric numerical integration lays preservation of qualitative behaviour of solutions, in this section we are concerned with the preservation of major qualitative features of the nonlinear \schr equation. Specifically, we wish to ascertain the preservation of $\CC{L}_2$ , momentum and Hamiltonian energy  under  iterated linearization, as described in subsection \ref{sec:algorithm}, once  they are applied to equation (\ref{eq:NLS}).

Throughout this section we assume that the NLS (\ref{eq:NLS}) is solved in a cube in $\BB{R}^d$ with periodic boundary conditions. A generalisation to zero Dirichlet or  Neumann boundary conditions is straightforward. 

\subsection{$\CC{L}_2$ energy preservation}

The $\CC{L}_2$ energy of the cubic NLS equation (\ref{eq:dispersive}) with periodic boundary conditions is conserved and a trivial -- yet pleasing -- feature of our method is that so does its numerical solution by means of the method (\ref{eq:IterativeScheme}). 

\begin{theorem}
  \label{Conserve_mass}
  Subject to periodic boundary conditions it is true that $\|u^{[k]}(t)\|=\|u_0(t)\|$ and the method (\ref{eq:IterativeScheme}) preserves $\CC{L}_2$ energy.
\end{theorem}

\begin{proof}
  Each step of  (\ref{eq:IterativeScheme}) involves a solution of a linear Schr\"odinger equation (\ref{eq:lin_NLS}), which itself preserves $\CC{L}_2$ energy subject to the above boundary conditions. This remains true once the solution of (\ref{eq:lin_NLS}) is itself computed by the methods of Section~2, because they also preserve $\CC{L}_2$ energy.
\end{proof}

Realistic application of Theorem~4 requires for the linear equation to be discretised and, to retain the spirit of the theorem,  the linearised equation must discretised by a time-stepping method that respects $\CC{L}_2$ norm conservation, e.g.\ the methods from \cite{iserles18mlm,iserles19sse}.

One crucialconsequence of the theorem is that the iterated linearisation \R{eq:IterativeScheme} is stable, provided that the linear equation is solved stably. 

\subsection{Conservation of momentum}

Given boundary conditions either in an interval  with periodic boundary conditions or in $\BB{R}$, the NLS (\ref{eq:dispersive}) conserves the momentum
\begin{displaymath}
  I(u)=\ii \int [\bar{u}_x(x) u(x)-\bar{u}(x) u_x(x)]\D x
\end{displaymath}
\cite{killip09cns}. Therefore, in a periodic setting in $\BB{T}_1$,
\begin{Eqnarray*}
  &&\frac{\D}{\D t}I(u^{[k]})\\
  &=&\ii \int_{\bb{T}^1} (\bar{u}^{[k]}_{xt} u^{[k]}+\bar{u}^{[k]}_x u_t^{[k]} -\bar{u}_t^{[k]} u_x-\bar{u}^{[k]}u_{xt}^{[k]})\D x\\
  &=&\ii\int_{\bb{T}^1} \{\ii[-u_{xx}^{[k]}+(\bar{u}_{xx}^{[k]}+\bar{u}_x^{[k-1]}u^{[k-1]}\bar{u}^{[k]} +\bar{u}^{[k-1]} u_x^{[k-1]}\bar{u}^{[k]}+|u^{[k-1]}|^2\bar{u}^{[k]}_x] u^{[k]}\\
  &&\hspace*{10pt}\mbox{}+\ii\bar{u}_x^{[k]}(u_x^{[k]}-|u^{[k-1]}|^2u^{[k]}-(-\bar{u}^{[k]}_x+|u^{[k-1]}|^2\bar{u}^{[k]})u^{[k]}_x\\
  &&\hspace*{10pt}\mbox{}-\ii \bar{u}^{[k]}[u_{xx}^{[k]}-(-u_{xx}^{[k]}+\bar{u}_x^{[k-1]}u^{[k-1]}u^{[k]} +\bar{u}^{[k-1]}u_x^{[k-1]}u^{[k]}+|u^{[k-1]}|^2 u_x^{[k]}]\}\D x\\
  &=&\int_{\bb{T}^1} (\bar{u}^{[k]}_{xx}u^{[k]}-\bar{u}^{[k]}u_{xx}^{[k]})\D x
 \end{Eqnarray*}
 and, integrating by parts,
\begin{displaymath}
  \frac{\D}{\D t}I(u^{[k]})=\int_{\bb{T}^1} (\bar{u}_x^{[k]}u_x^{[k]}- \bar{u}_x^{[k]}u_x^{[k]})\D x=0.
\end{displaymath}

This can be easily generalised componentwise to $\BB{T}^d$ or $\BB{R}^d$ and we deduce that

\begin{theorem}
 The method \R{eq:IterativeScheme} conserves momentum.
\end{theorem}

\subsection{Hamiltonian energy}

Hamiltonian energy of cubic  \schr equation in our setting is given by
\begin{equation}\label{energy_1}
\mathcal{H}(u)=\frac{1}{(2\pi)^d}\int_{\bb T^d}\left[\|{\nabla}u\|_{\CC{L}_2}^2+\frac12 \lambda |u|^4\right]\D x=\langle -\Delta u,u\rangle+\frac{\lambda}{2}
\langle |u|^2u,u\rangle
\end{equation}
and is preserved in time \cite{faou12gni}. For completeness, we prove this statement: it is enough to show that its time derivative 
$
\frac{\D}{\D t}\mathcal{H}(u)=\frac{\D}{\D t}\langle -\Delta  u,u \rangle + \frac{\D}{\D t}\frac{1}{2} \langle |u|^2u,u \rangle
$
equals  zero. Subject to the notational convention that $\dot{u}=\frac{\D}{\D t}u$ we observe that
\begin{align}\label{kinetic}
\frac{\D}{\D t}\langle -\Delta  u,u \rangle &= -\langle \Delta  \dot{u},u \rangle -\langle \Delta  u,\dot{u} \rangle = -\langle   \dot{u}, \Delta u \rangle -\overline{\langle    \dot{u} , \Delta u \rangle}=-2{\rm Re}\langle \dot{u}, \Delta u  \rangle \\ \nonumber
&=-2{\rm Re}\langle  \ii \Delta u -\ii |u|^2 u, \Delta u  \rangle=2{\rm Re}\langle  \ii |u|^2 u, \Delta u  \rangle \\
\label{potential}
\frac{1}{2}\frac{\D}{\D t} \langle |u|^2u,u \rangle &= \frac{1}{2}(\langle \dot{u}\bar{u}u,u \rangle+\langle u\dot{\bar{u}}u,u \rangle+\langle u\bar{u}\dot{u},u \rangle+\langle u\bar{u}u,\dot{u} \rangle) \\ \nonumber
&=\frac{1}{2}(2\langle \dot{u}\bar{u}u,u \rangle+2\langle u\bar{u}u,\dot{u} \rangle)=2\langle \dot{u},u\bar{u}u\rangle+\langle u\bar{u}u,\dot{u} \rangle\\\nonumber
&=\langle \dot{u},u|u|^2 \rangle+\langle u|u|^2,\dot{u} \rangle \\ \nonumber
&=2{\rm Re}\langle u|u|^2,\dot{u} \rangle=2{\rm Re}\langle u|u|^2,\ii\Delta u-\ii|u|^2u \rangle=2{\rm Re}\langle u|u|^2,\ii\Delta u \rangle\\
\nonumber
&=-2{\rm Re}\langle \ii u|u|^2,\Delta u \rangle
\end{align}

It is evident from (\ref{kinetic}) and (\ref{potential}) that $\frac{\D}{\D t}\mathcal{H}(u)=0$ so the Hamiltonian energy of cubic \schr equation is preserved. The same cannot be concluded for $u^{[k]}$, the solution of $k$th linearised iteration
\begin{equation}
  \label{eq:linearized}
  \ii u^{[k]}_t=-\Delta u^{[k]}+\lambda |u^{[k-1]}|^2u^{[k]}.
\end{equation}
Indeed, given that the Hamiltonian of the linear Schr\"odinger equation $\ii u_t=-\Delta u+V(x)u$ is
\begin{displaymath}
  \int_{\bb{T}^d} \left[\|\nabla u\|^2_{\CC{L}_2} +V(x)|u|^2\right]\!\D x,
\end{displaymath}
all we can deduce from (\ref{eq:linearized}) is that
\begin{displaymath}
  \int_{\bb{T}^d} \left[\|\nabla u^{[k]}\|_{\CC{L}_2}^2+\lambda |u^{[k-1]}|^2|u^{[k]}|^2\right]\!\D x
\end{displaymath}
is conserved under the flow and, at a first glance, a factor of $\frac12$ is missing. Fortunately, it is recovered once $k$ tends to infinity and the following theorem establishes the rate of convergence.

\begin{theorem}
 Once iterated linearisation is applied to the NLS (\ref{eq:NLS}), it is true that  
 \begin{equation}
  \label{eq:Hammy}
  \mathcal{H}(u^{[k]})-\mathcal{H}(u)=\O{h^{k+1}},\qquad k\in\BB{N}.
\end{equation}
\end{theorem}

\begin{proof}
  We commence by recalling from Theorem 3 that $e^{[k]}(t)=u^{[k]}(t)-u(t)=\O{h^{k+1}}$ and observing that (\ref{energy_1}) is equivalent to
  \begin{displaymath}
    \mathcal{H}(u)=\left\langle \ii u_t-\frac{\lambda}{2}|u|^2u,u\right\rangle\!.
  \end{displaymath}
  Moreover,
  \begin{Eqnarray*}
    \mathcal{H}(u^{[k]})&=&\left\langle \ii u^{[k]}_t-\frac{\lambda}{2}|u^{[k]}|^2u^{[k]},u^{[k]}\right\rangle\\
    &=&\left\langle \ii u_t+\ii e^{[k]}_t-\frac{\lambda}{2}|u+e^{[k]}|^2(u+e^{[k]}),u+e^{[k]}\right\rangle
  \end{Eqnarray*}
  Since $|u+e^{[k]}|^2=|u|^2+2\Re u\bar{e}^{[k]}+\O{h^{2k+2}}$, we have
  \begin{Eqnarray*}
    \mathcal{H}(u^{[k]})&=&\mathcal{H}(u) +\left\langle \ii u_t-\frac{\lambda}{2}|u|^2u,e^{[k]}\right\rangle\\
    &&\mbox{}+\left\langle \ii u_t-\lambda (\Re u\bar{e}^{[k]})(u+e^{[k]})-\frac{\lambda}{2}|e^{[k]}|^2(u+e^{[k]}),u\right\rangle+\O{h^{2k+2}}.
  \end{Eqnarray*}
  Therefore
  \begin{Eqnarray*}
    \left|\mathcal{H}(u^{[k]})-\mathcal{H}(u) \right|&\leq& \left\|\ii u_t-\frac{\lambda}{2}|u|^2u\right\| \|e^{[k]}\|+|\lambda|\|u+e^{[k]}\| |\Re u\bar{e}^{[k]}|\\
    &&\mbox{}+\frac{|\lambda|}{2} \|u\|\cdot\|u+e^{[k]}\|\cdot\|e^{[k]}\|^2+|\langle u_t,u\rangle|+\O{h^{k+1}},
  \end{Eqnarray*}
  since
  \begin{displaymath}
    \langle u_t,u\rangle=\frac12 \partial_t \|u\|^2=0,
  \end{displaymath}
  because of the $\CC{L}_2$ energy conservation. Thus, $\|e^{[k]}\|=\O{h^{k+1}}$ implies  (\ref{eq:Hammy}) and we are done.
\end{proof}

\setcounter{equation}{0}
\setcounter{figure}{0}
\section{Numerical examples}
\label{sec:numerics}
In this section we present numerical evidence to
support the theoretical developments.
Our approach of iterative linearisation using Magnus expansions in conjunction with Hermite quadratures is denoted by MH throughout this section.
We combine it with two different methods for computing the matrix exponential: splitting methods and Krylov subspace methods.

\subsection{MH with splitting methods}



In this section, we consider splitting methods for the exponentiation of the linearised Magnus expansion. We note that {\em classical} splitting methods such as the Strang splitting \cite{strang1968construction} and Blanes--Moan splittings \cite{blanes2002practical}, referred to as BM in this section, can be applied directly to nonlinear and non-autonomous equations.

It is possible to utilise such splittings for the exponentiation of the central term in \R{eq:MagnusAssymetricSplit}, the commutator-eliminated form of our iteratively linearised Magnus-Hermite (MH) scheme. The combination of MH and BM is termed MHBM in the sequel.
When splittings methods can be directly applied to the nonlinear equations, they typically perform better, as seen by comparing BM vs MHBM in Figure \ref{fig: positive nonlinear, time-dependent convergence plot}.

However, a large class of splittings that feature commutators, such as {\em compact} splittings \cite{chin05fsi,omelyan2003symplectic}, {\em asymptotic} splittings such as symmetric Zassenhaus splittings \cite{bader14eas} Magnus--Zassenhaus splitting \cite{iserles2019solving}, and specialised splittings for laser-matter interaction \cite{iserles2019compact} have been developed in recent years, and there is no straightforward way of extending their application to nonlinear equations.
We demonstrate that out iterative linearisation procedure allows the use of such splittings for nonlinear equations. In particular, we combine our Magnus-Hermite (MH) schemes with the Chin--Chen splitting \cite{chin05fsi}, after the elimination of the commutators \R{eq:MagnusAssymetricSplit}. This method is referred to as MHC in the sequel.

We consider the following examples, where the Hamiltonian is given by
\begin{equation}
\label{eq: 1D example}
\ii \partial_t u(x,t) = -\partial^2_x u(x,t) + V(u, x, t)  u(x,t), \qquad u(x,0) = u_0(x) \in \BB{H},\quad t \geq 0, \quad \lambda \in \BB{R},
\end{equation}
where the total potential $V(u, x, t)$ is a function of $u(x,t)$, $x$, and $t$, containing a static potential $V^0(x)$, an external field $V^{\mathrm{e}}(x,t)$, and a nonlinear potential $V^{\mathrm{nl}}$:
\[ V(u, x, t) = V^0(x) + V^{\mathrm{e}}(x,t) + V^{\mathrm{nl}}(u) ,\]
where
\[ V^{\mathrm{nl}}(u) = \lambda  |u|^2 .\]
We chose the static potential
\[ V^0(x) = x^4 - 10x^2, \]
the external field
\[ V^{\mathrm{e}}(x,t) = 5  \sin(5\pi t)  \sin(\pi x), \]
and the initial condition as the as wave packet
\[ u_0(x) = \exp\left(-\frac{(x - x_0)^2}{2\sigma^2}\right) ,\]
where $ x_0 = -2.0 $ and $ \sigma^2 = 0.25 $. We consider the solution on a periodic spatial grid $[-10,10]$, discretised with $1000$ grid points, and use a normalised version of $u$ as an initial condition $u_0$.

\begin{example}[GP, defocusing, with external time-dependent field]
    \normalfont
    \label{ex: positive nonlinear, time-dependent}
    In this example, we consider the potential
    \[ V(u, x, t) = V^0(x) + V^{\mathrm{e}}(x,t) + V^{\mathrm{nl}}(u) , \qquad \text{with}\ \lambda = 10\ \text{in}\ V^{\mathrm{nl}},\]
    with static potential $V^0(x)$,
    external field $V^{\mathrm{e}}(x,t)$, and nonlinear potential $V^{\mathrm{nl}}$.

    The evolution of solution is shown in Figure \ref{fig: positive nonlinear, time-dependent solutions plot}, while the convergence of the $L^2$ error of $u$ at $t = 1$ computed by different methods is shown in Figure \ref{fig: positive nonlinear, time-dependent convergence plot}. Observables over time are shown in Figure \ref{fig: observables over time plot}, and the change in observables with time is shown in Figure \ref{fig: change in observables plot}.

    \begin{figure}[H]
        \centering
            \resizebox{\textwidth}{!}{\input{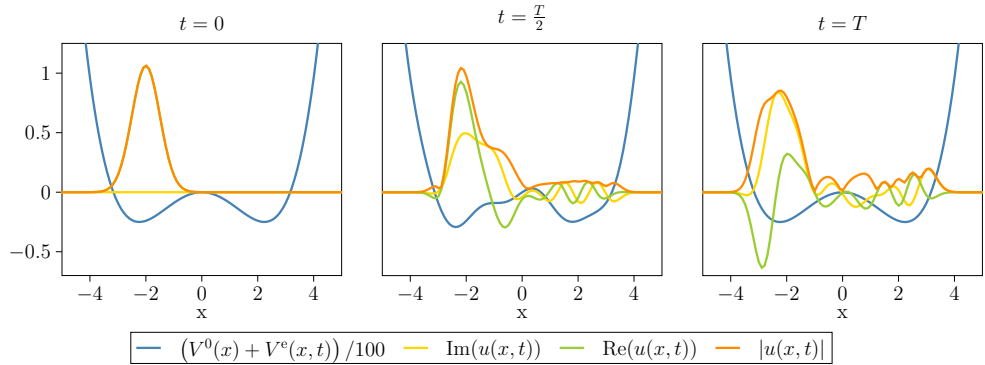}}
            \caption{The linear part of the potential, $V^0(x) + V^{\mathrm{e}}(x,t)$ (scaled down by a factor of $100$ for ease of visualisation), real and imaginary parts and absolute value of the wave function $u(x,t)$ at initial time $t = 0$ (left), intermediate time $t =T/2$ (middle) and final time $t = T$ (right), where $T = 1$.
            }
            \label{fig: positive nonlinear, time-dependent solutions plot}
    \end{figure}

    \begin{figure}[H]
        \centering
            \resizebox{\textwidth}{!}{
\begin{tikzpicture}

\definecolor{crimson2143940}{RGB}{214,39,40}
\definecolor{darkgray176}{RGB}{176,176,176}
\definecolor{darkorange25512714}{RGB}{255,127,14}
\definecolor{forestgreen4416044}{RGB}{44,160,44}
\definecolor{gray}{RGB}{128,128,128}
\definecolor{lightgray204}{RGB}{204,204,204}
\definecolor{steelblue31119180}{RGB}{31,119,180}

\begin{axis}[
legend cell align={left},
legend style={
  fill opacity=0.8,
  draw opacity=1,
  text opacity=1,
  at={(1.05,1)},
  anchor=north west,
  draw=lightgray204
},
log basis x={10},
log basis y={10},
tick align=outside,
tick pos=left,
unbounded coords=jump,
x grid style={darkgray176},
xlabel={timestep size (h)},
xmin=0.000794328234724281, xmax=0.125892541179417,
xmode=log,
xtick style={color=black},
xtick={1e-05,0.0001,0.001,0.01,0.1,1,10},
xticklabels={
  \(\displaystyle {10^{-5}}\),
  \(\displaystyle {10^{-4}}\),
  \(\displaystyle {10^{-3}}\),
  \(\displaystyle {10^{-2}}\),
  \(\displaystyle {10^{-1}}\),
  \(\displaystyle {10^{0}}\),
  \(\displaystyle {10^{1}}\)
},
y grid style={darkgray176},
ylabel={$L^2$ error at T = 1},
ymin=6.66261822360617e-12, ymax=4.88149216703617,
ymode=log,
ytick style={color=black},
ytick={1e-14,1e-12,1e-10,1e-08,1e-06,0.0001,0.01,1,100,10000},
yticklabels={
  \(\displaystyle {10^{-14}}\),
  \(\displaystyle {10^{-12}}\),
  \(\displaystyle {10^{-10}}\),
  \(\displaystyle {10^{-8}}\),
  \(\displaystyle {10^{-6}}\),
  \(\displaystyle {10^{-4}}\),
  \(\displaystyle {10^{-2}}\),
  \(\displaystyle {10^{0}}\),
  \(\displaystyle {10^{2}}\),
  \(\displaystyle {10^{4}}\)
}
]
\addplot [thick, steelblue31119180]
table {%
0.1 0.522414983659632
0.03125 0.0325888880380696
0.01 0.00305613942968701
0.00315457413249211 0.000303023814785296
0.001 3.04398912564862e-05
};
\addlegendentry{Strang}
\addplot [thick, darkorange25512714]
table {%
0.1 0.0165534136129378
0.03125 0.000337719918953307
0.01 1.01947509427572e-06
0.00315457413249211 2.2972921684453e-09
0.001 2.30652896773097e-11
};
\addlegendentry{BM}
\addplot [thick, forestgreen4416044]
table {%
0.1 nan
0.03125 nan
0.01 1.85717704073659e-05
0.00315457413249211 1.01457412176613e-07
0.001 1.02136849083069e-09
};
\addlegendentry{MHC}
\addplot [thick, crimson2143940, dashed]
table {%
0.1 1.41006330835206
0.03125 1.34009429813865
0.01 1.08982111175528e-05
0.00315457413249211 1.07205764947415e-07
0.001 1.08047609927121e-09
};
\addlegendentry{MHBM}
\addplot [thick, black, dash pattern=on 1pt off 3pt on 3pt off 3pt]
table {%
0.1 1
0.03125 0.09765625
0.01 0.01
0.00315457413249211 0.000995133795738837
0.001 0.0001
};
\addlegendentry{$O(h^2)$}
\addplot [thick, black, dashed]
table {%
0.1 0.01
0.03125 9.5367431640625e-05
0.01 1e-06
0.00315457413249211 9.90291271421585e-09
0.001 1e-10
};
\addlegendentry{$O(h^4)$}
\addplot [thin, gray, dotted, forget plot]
table {%
0.0999999999999999 6.66261822360617e-12
0.0999999999999999 4.88149216703616
};
\addplot [thin, gray, dotted, forget plot]
table {%
0.00899999999999999 6.66261822360617e-12
0.00899999999999999 4.88149216703616
};
\addplot [thin, gray, dotted, forget plot]
table {%
0.008 6.66261822360617e-12
0.008 4.88149216703616
};
\addplot [thin, gray, dotted, forget plot]
table {%
0.007 6.66261822360617e-12
0.007 4.88149216703616
};
\addplot [thin, gray, dotted, forget plot]
table {%
0.006 6.66261822360617e-12
0.006 4.88149216703616
};
\addplot [thin, gray, dotted, forget plot]
table {%
0.005 6.66261822360617e-12
0.005 4.88149216703616
};
\addplot [thin, gray, dotted, forget plot]
table {%
0.004 6.66261822360617e-12
0.004 4.88149216703616
};
\addplot [thin, gray, dotted, forget plot]
table {%
0.003 6.66261822360617e-12
0.003 4.88149216703616
};
\addplot [thin, gray, dotted, forget plot]
table {%
0.002 6.66261822360617e-12
0.002 4.88149216703616
};
\addplot [thin, gray, dotted, forget plot]
table {%
0.001 6.66261822360617e-12
0.001 4.88149216703616
};
\addplot [thin, gray, dotted, forget plot]
table {%
0.09 6.66261822360617e-12
0.09 4.88149216703616
};
\addplot [thin, gray, dotted, forget plot]
table {%
0.0799999999999999 6.66261822360617e-12
0.0799999999999999 4.88149216703616
};
\addplot [thin, gray, dotted, forget plot]
table {%
0.07 6.66261822360617e-12
0.07 4.88149216703616
};
\addplot [thin, gray, dotted, forget plot]
table {%
0.06 6.66261822360617e-12
0.06 4.88149216703616
};
\addplot [thin, gray, dotted, forget plot]
table {%
0.05 6.66261822360617e-12
0.05 4.88149216703616
};
\addplot [thin, gray, dotted, forget plot]
table {%
0.04 6.66261822360617e-12
0.04 4.88149216703616
};
\addplot [thin, gray, dotted, forget plot]
table {%
0.03 6.66261822360617e-12
0.03 4.88149216703616
};
\addplot [thin, gray, dotted, forget plot]
table {%
0.02 6.66261822360617e-12
0.02 4.88149216703616
};
\addplot [thin, gray, dotted, forget plot]
table {%
0.01 6.66261822360617e-12
0.01 4.88149216703616
};
\addplot [thin, gray, dotted, forget plot]
table {%
0.000794328234724281 1e-11
0.125892541179416 1e-11
};
\addplot [thin, gray, dotted, forget plot]
table {%
0.000794328234724281 1e-10
0.125892541179416 1e-10
};
\addplot [thin, gray, dotted, forget plot]
table {%
0.000794328234724281 1e-09
0.125892541179416 1e-09
};
\addplot [thin, gray, dotted, forget plot]
table {%
0.000794328234724281 1e-08
0.125892541179416 1e-08
};
\addplot [thin, gray, dotted, forget plot]
table {%
0.000794328234724281 1e-07
0.125892541179416 1e-07
};
\addplot [thin, gray, dotted, forget plot]
table {%
0.000794328234724281 1e-06
0.125892541179416 1e-06
};
\addplot [thin, gray, dotted, forget plot]
table {%
0.000794328234724281 1e-05
0.125892541179416 1e-05
};
\addplot [thin, gray, dotted, forget plot]
table {%
0.000794328234724281 0.0001
0.125892541179416 0.0001
};
\addplot [thin, gray, dotted, forget plot]
table {%
0.000794328234724281 0.001
0.125892541179416 0.001
};
\addplot [thin, gray, dotted, forget plot]
table {%
0.000794328234724281 0.01
0.125892541179416 0.01
};
\addplot [thin, gray, dotted, forget plot]
table {%
0.000794328234724281 0.1
0.125892541179416 0.1
};
\addplot [thin, gray, dotted, forget plot]
table {%
0.000794328234724281 1
0.125892541179416 1
};
\end{axis}

\end{tikzpicture}}
            \caption{$L^2$ error in $u(T)$ for Example~\ref{ex: positive nonlinear, time-dependent}}
            \label{fig: positive nonlinear, time-dependent convergence plot}
    \end{figure}
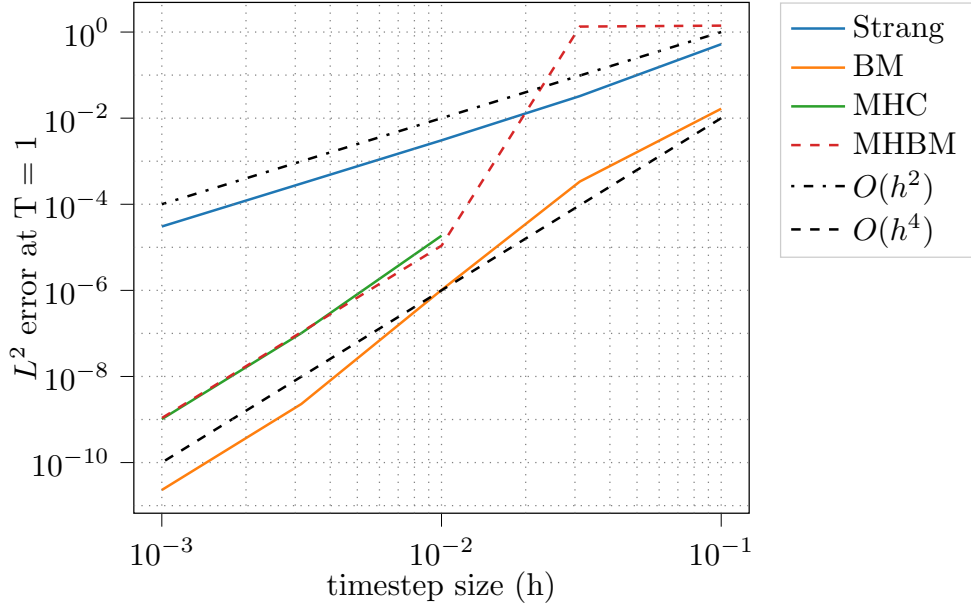

\end{example}

\begin{example}[GP, defocusing, without external time-dependent field]
    \normalfont
    \label{ex: positive nonlinear, time-independent}
    In this example, we consider the potential
    \[ V(u, x, t) = V^0(x) + V^{\mathrm{nl}}(u) , \qquad \text{with}\ \lambda = 10\ \text{in}\ V^{\mathrm{nl}},\]
    while the external field $V^{\mathrm{e}}(x,t)$ is absent.
\end{example}

\begin{example}[GP, focusing, without external time-dependent field]
    \normalfont
    \label{ex: negative nonlinear, time-independent}
    In this example, we consider the potential
    \[ V(u, x, t) = V^0(x) + V^{\mathrm{nl}}(u) , \qquad \text{with}\ \lambda = -10\ \text{in}\ V^{\mathrm{nl}},\]
    while the external field $V^{\mathrm{e}}(x,t)$ is absent.
\end{example}

\begin{example}[NLS, defocusing]
    \normalfont
    \label{ex: positive nonlinear part only}
    In this example, we consider the potential
    \[ V(u, x, t) = V^{\mathrm{nl}}(u), \qquad \text{with}\  \lambda = 10\ \text{in}\ V^{\mathrm{nl}},\]
    i.e., with the nonlinear potential only.
\end{example}

\begin{figure}[H]
    \centering
        \resizebox{\textwidth}{!}{\input{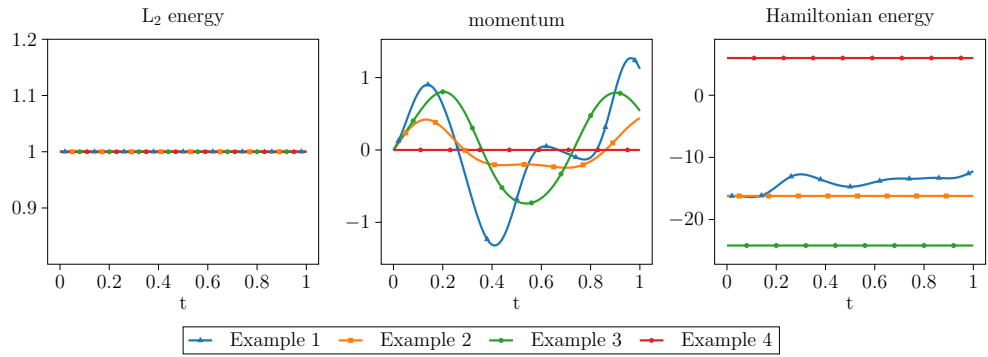}}
        \caption{Observables over time for Examples \ref{ex: positive nonlinear, time-dependent}, \ref{ex: positive nonlinear, time-independent}, \ref{ex: negative nonlinear, time-independent} and \ref{ex: positive nonlinear part only}.}
        \label{fig: observables over time plot}
\end{figure}

\begin{figure}[H]
    \centering
        \resizebox{\textwidth}{!}{\input{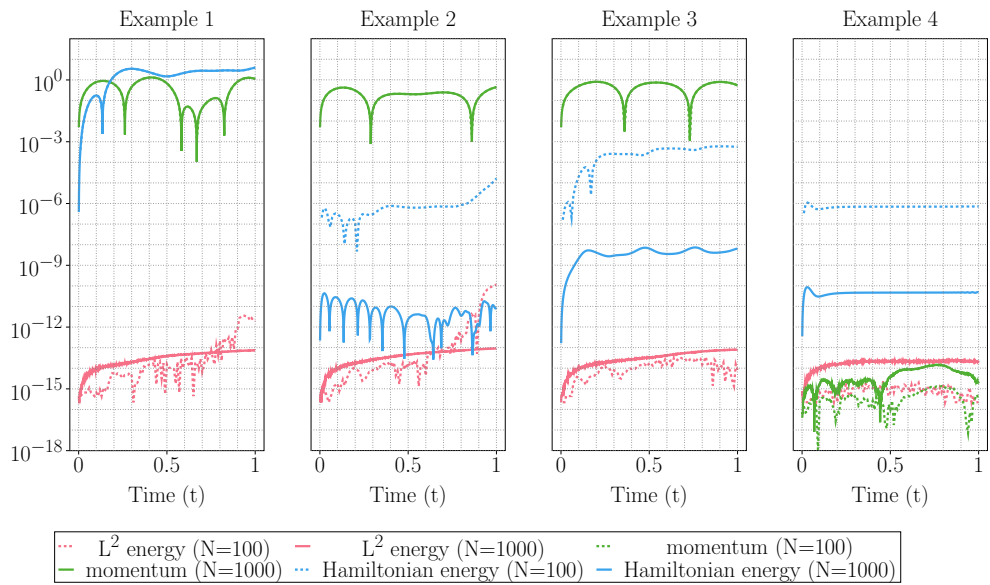}}
        \caption{Change in observables for Examples \ref{ex: positive nonlinear, time-dependent}, \ref{ex: positive nonlinear, time-independent}, \ref{ex: negative nonlinear, time-independent} and \ref{ex: positive nonlinear part only} over time, where $N$ is the number of time steps.}
        \label{fig: change in observables plot}
\end{figure}

Observables over time for all the examples computed with the MHC method are shown in Figure \ref{fig: observables over time plot}. The $\CC{L}^2$ energy is always preserved. The momentum is preserved only in case of the NLS, i.e, Example~\ref{ex: positive nonlinear part only}. The Hamiltonian energy is preserved up to $\O{h^4}$ except in the scenario presented in Example~\ref{ex: positive nonlinear, time-dependent}, which involves an external time-independent field. These behaviours are in line with the discussion in Section~\ref{sec:energy}. The changes in observables with time are shown in Figure \ref{fig: change in observables plot}.

\subsection{MH with a Krylov subspace method}

In this subsection, we consider an example where a splitting approach is not feasible. For example, when the Hamiltonian is a large sparse matrix, the Krylov method is particularly useful. Consider
\begin{equation}
\label{eq: matrix operator example}
\ii \partial_t u(t) = \mathcal{L}(t) u(t) + V^{\mathrm{nl}}( u(t))  u(t), \qquad u(0) = u_0 \in \BB{H},\quad t \geq 0, \quad \lambda \in \BB{R},
\end{equation}
where the linear self-adjoint operator $\mathcal{L}(t)$ can be decomposed into time-independent
and time-dependent components, $\mathcal{L}(t) = \mathcal{L}_0 + \mathcal{L}_1(t)$, where $\mathcal{L}_0 = L_0 $, and $\mathcal{L}_1(t) = V^{\mathrm{e}}(t) L_1$, with $L_0$ and $L_1$ being $n \times n$ matrices for $u$ discretised on $n$ points.

We combine our iteratively linearised Magnus-Hermite  expansion \eqref{eq:MagnusO2simplified} without elimination of commutators with a Lanczos approximation of the matrix exponential \cite{park1986unitary,saad2003iterative,iserles2018magnus}. The Lanczos method is a Krylov subspace method that is well suited to Hermitian and skew-Hermitian matrices and is one of the most commonly applied methods in quantum dynamics. We use the $\texttt{mkprop}$ package \cite{mkprop} for the implementation of the Lanczos method, with an accuracy $\texttt{tol}= 10^{-8}$. The overall scheme is denoted MHK.

We use randomly generated $n \times n$ Hermitian matrix $L_0$ and $L_1$,
and defocusing cubic nonlinear potential
\[V^{\mathrm{nl}}(u) = |u|^2.\]

\begin{example}[nonlinear, with external time-dependent field]
    \normalfont
    \label{ex: matrix, nonlinear, time-dependent}
We consider external time-dependent field
\[V^{\mathrm{e}}(t) =  \sin(5\pi t).\]

\end{example}




\begin{example}[nonlinear, without external time-dependent field]
    \normalfont
    \label{ex: matrix, nonlinear, time-independent}
    We consider no external time-dependent field
\[V^{\mathrm{e}}(t) \equiv 0.\]


\end{example}



\begin{figure}[H]
    \centering
        \resizebox{\textwidth}{!}{\input{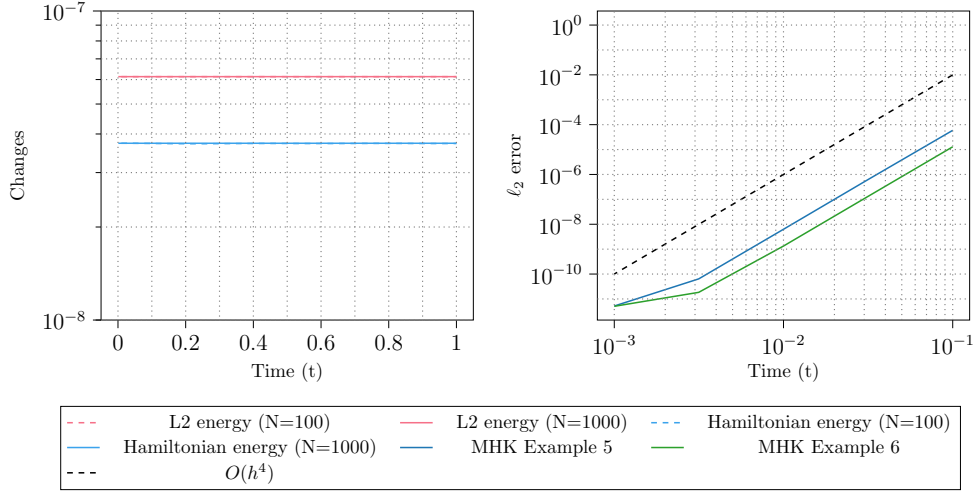}}
        \caption{Left: change of observables over time for Example \ref{ex: matrix, nonlinear, time-independent}. Right: convergence of the $\ell_2$ error of $u$ at $t = 1$ computed by different methods for Examples \ref{ex: matrix, nonlinear, time-dependent}, \ref{ex: matrix, nonlinear, time-independent}.}
        \label{fig: Krylov plots in one}
\end{figure}

In Examples~\ref{ex: matrix, nonlinear, time-dependent} and \ref{ex: matrix, nonlinear, time-independent}, the Hamiltonian energy $\langle u(t), \mathcal{L}_0 u(t)\rangle$ is preserved as desired, as can be seen in Figure
\ref{fig: Krylov plots in one} (left).
The convergence of the $\ell_2$ error of $u$ at $t = 1$ computed by Krylov method for Examples~\ref{ex: matrix, nonlinear, time-dependent} and \ref{ex: matrix, nonlinear, time-independent} is displayed in Figure \ref{fig: Krylov plots in one} (right). 

\setcounter{equation}{0}
\setcounter{figure}{0}
\section{Conclusions}
\label{sec:conclusions}

The main theme of this paper is an iterated linearisation of dispersive equations \R{eq:dispersive}, converting them into the solution of simpler systems \R{eq:IterativeScheme}. In particular, in the case of NLS \R{eq:NLS}, we convert a nonlinear problem into an iterated solution of significantly simpler linear Schr\"odinger equations, for which very effective and cheap methods are available. We proved that, in the case of NLS, this approach conserves  $\CC{L}_2$ energy and approximates well the Hamiltonian. 

However, there is nothing special in dispersive equations or in the form \R{eq:dispersive}. Our approach applies to numerous PDEs of evolution. An example of a dispersive equation that does not fit into the \R{eq:dispersive} pattern is the Korteveg--de Vries (KdV) equation
\begin{equation}
  \label{eq:KdV}
  \partial_t u=-\partial_x^3 u+6u\partial_x u,
\end{equation}
with an initial condition $u=u_0$ and suitable boundary conditions. An iterated linearisation is
\begin{Eqnarray}
  \label{eq:discKdV}
  u^{[0]}&=&u_0,\\
  \nonumber
  \partial_t u^{[k+1]}&=&-\partial_x^3 u^{[k+1]}+6(\partial_x u^{[k]}) u^{[k+1]},\qquad k\in\BB{Z}_+,
\end{Eqnarray}
and it relies on the fact that the linear equation
\begin{displaymath}
  \partial_t u=-\partial_x^3u+V(t,x)u,
\end{displaymath}
where $u$ is known, is considerably easier to solve.

Iterated linearisation can be also applied to non-dispersive PDEs. Consider, for example, the parabolic {\em Fitzhugh--Nagumo system\/}
\begin{Eqnarray}
  \label{eq:FitzhughNagumo}
  \partial_t v&=&d_v\Delta v+f(v,w),\\
  \nonumber
  \partial_t w&=& d_w\Delta w+g(u,w)
\end{Eqnarray}
from mathematical biology \cite{izhikevich07dsn}, given with an initial condition $v=v_0$, $w=w_0$ and Dirichlet boundary conditions. $f$ and $g$ are given functions, typically low-degree polynomials, while $d_v,d_w>0$ are given constants. Provided that $f(v,w)=\tilde{f}(v,w)v$ and $g(v,w)=\tilde{g}(v,w)w$, we may consider the iterated linearisation
\begin{Eqnarray}
  \label{eq:FN_iteration}
  v^{[0]}&=&v_0,\qquad w^{[0]}=w_0,\\
  \nonumber
  \partial_t v^{[k+1]}&=&d_v\Delta v^{[k+1]}+\tilde{f}(v^{[k]},w^{[k]})v^{[k+1]},\\
  \nonumber
  \partial_t w^{[k+1]}&=&d_w \Delta w^{[k+1]}+\tilde{g}(v^{[k]},w^{[k]})w^{[k+1]},\qquad k\in\BB{Z}_+.
\end{Eqnarray}
The computation of a diffusion equation is an easy numerical endeavour and this can be leveraged toward an effective iteration of the form \R{eq:FN_iteration}. 

Note that both \R{eq:discKdV} and \R{eq:FN_iteration} are merely at the level of preliminary ideas, without any stability or convergence analysis. Moreover, both our examples, KdV \R{eq:KdV} and the Fitzhugh--Nagumo system \R{eq:FitzhughNagumo}, possess a wealth of qualitative features which remain unexplored in the context of iterated linearisation. We mention the two equations mostly to illustrate the potential scope of our approach beyond the realm of dispersive equations. 

\section*{Acknowledgements}

The work of GC in this project has been supported by EPSRC (EP/S022945/1). The work of KK in this project has been supported by The National Center for Science (NCN), based on Grant No.\ 2019/34/E/ST1/00390.

\bibliographystyle{agsm}
\bibliography{refs}

\end{document}